\documentclass[11pt]{amsart}
\usepackage[height=8.2in,width=5.6in,margin=1.2in]{geometry}
\usepackage{amsfonts, mathdots,color}
\usepackage{amssymb,latexsym}
\usepackage{enumerate}
\usepackage[colorlinks=true,citecolor=blue]{hyperref}
\usepackage{bm}
\makeatletter
\@namedef{subjclassname@2010}{%
  \textup{2010} Mathematics Subject Classification}
\makeatother

\newtheorem{thm}{Theorem}[section]
\newtheorem{cor}[thm]{Corollary}
\newtheorem{lem}[thm]{Lemma}

\theoremstyle{definition}
\newtheorem{defin}[thm]{Definition}

\numberwithin{equation}{section}

\DeclareMathOperator*{\sgn}{sgn}

\DeclareMathOperator{\mdet}{det}

\newcommand{\C}{\mathbb{C}}
\newcommand{\dxg}{d_\chi^G}

\begin{document}


\title[Generalized matrix functions]{Completely strong superadditivity of generalized matrix functions}

\author[M. Lin]{Minghua Lin}
\address{Department of Mathematics and Statistics\\
University of Victoria\\
 Victoria, BC, Canada, V8W 3R4.
}
\email{mlin87@ymail.com}

\author[S. Sra]{Suvrit Sra}\address{MPI for Intelligent Systems\\ T\"ubingen, Germany}\email{suvrit@gmail.com}

\date{}

\begin{abstract} We prove that generalized matrix functions  satisfy a block-matrix strong superadditivity inequality over the cone of  positive semidefinite matrices. Our result extends a recent result  of Paksoy-Turkmen-Zhang  \cite{PTZ14}. As an application, we obtain a short proof of a classical inequality of Thompson (1961) on block matrix determinants.

\end{abstract}



\maketitle

\section{Introduction}
Let $\mathbb{M}_n$ denote the algebra of all $n\times n$ complex matrices.  Let $\mathcal{A}\subset \mathbb{M}_n$. A functional $f: \mathcal{A} \to \mathbb{R}$ is called \emph{superadditive} if for all $A, B \in \mathcal{A}$
\begin{equation*}
  f(A+B)\ge f(A)+f(B),
\end{equation*}
and it is called \emph{strongly superadditive} if for all $A, B, C\in \mathcal{A}$
\begin{equation*}
  f(A+B+C)+f(C)\ge f(A+C)+f(B+C).
\end{equation*}

It is  known (e.g., \cite[Eq.(5)]{TCL11}) that the determinant  is strongly superadditive (and so superadditive) over the cone of positive semidefinite matrices.  That is,  \begin{eqnarray}\label{superadd-det}
\det (A+B+C)+\det C\ge \det (A+C)+\det (B+C) \end{eqnarray} for $A, B, C\ge 0$.

\begin{defin}
  \label{def.dxg}
  Let $\chi$ be a character of the subgroup $G$ of the symmetric group $S_n$. The \emph{generalized matrix function} $\dxg: \mathbb{M}_n \to \C$ is defined by
  \begin{equation}
    \label{eq:1}
    \dxg(A) := \sum_{\sigma \in G} \chi(\sigma) \prod_{i=1}^n a_{i\sigma(i)},
  \end{equation}
  where $A=[a_{ij}]$.
\end{defin}
When $G=S_n$ and $\chi(\sigma)=\sgn(\sigma)$ then $\dxg(A)$ reduces to the determinant $\det(A)$, while for $\chi(\sigma)\equiv 1$ we obtain $\dxg(A)=\text{per}(A)$, the permanent of $A$.

Recently,  Paksoy, Turkmen and Zhang \cite{PTZ14} presented a natural extension of (\ref{superadd-det}) via an embedding approach and through tensor products. More precisely,  for $A, B, C\ge 0$ they proved
\begin{eqnarray}\label{superadd-gmf}
\dxg (A+B+C)+\dxg (C)\ge \dxg (A+C)+\dxg (B+C).
\end{eqnarray}

This paper extends the above-cited strong superadditivity results to block matrices, thereby obtaining ``completely strong superadditivity'' for generalized matrix functions.

Before stating our problem formally, let us fix some notation.
The conjugate transpose of $X\in\mathbb{M}_n$ is denoted by $X^*$. For Hermitian matrices $X, Y\in \mathbb{M}_n$, the inequality  $X\ge Y$ means $X-Y$ is positive semidefinite.
Let $\mathbb{M}_m(\mathbb{M}_n)$ be the algebra of $m\times m$ block matrices with each block in $\mathbb{M}_n$. We will denote members of $\mathbb{M}_m(\mathbb{M}_n)$ via bold letters such as $\bm{A}$. A map (not necessarily linear) $\phi: \mathbb{M}_n\to \mathbb{M}_k$ is
\emph{positive} if it maps positive semidefinite matrices to positive semidefinite matrices. This map is \emph{completely positive} if for each positive integer $m$, the blockwise map $\Phi:  \mathbb{M}_m(\mathbb{M}_n)\to  \mathbb{M}_m(\mathbb{M}_k)$  defined by
\begin{equation}
  \label{map}
  \Phi\Big([A_{i,j}]_{i,j=1}^m\Big)=[\phi(A_{i,j})]_{i,j=1}^m
\end{equation}
is positive. The determinant is well-known to be completely positive  \cite{Hua55}. More generally, it is known that the generalized matrix functions are completely positive (e.g., \cite[Theorem 3.1]{Zhang12}).

The following definition extends the notion of strong superadditivity.

\begin{defin} Let $\bm{A}=[A_{i,j}]_{i,j=1}^m, \bm{B}=[B_{i,j}]_{i,j=1}^m, \bm{C}=[C_{i,j}]_{i,j=1}^m\in  \mathbb{M}_m(\mathbb{M}_n)$  be Hermitian. A map  $\phi:   \mathbb{M}_n \to  \mathbb{M}_k$ is said to be \emph{completely strongly superadditive} (CSS) if for each positive integer $m$, the map $\Phi$ defined in (\ref{map}) satisfies $$\Phi(\bm{A}+\bm{B}+\bm{C})+\Phi(\bm{C})\ge \Phi(\bm{A}+\bm{C})+\Phi(\bm{B}+\bm{C}).$$  \end{defin}

Our main assertion in this paper is as follows.
\begin{thm}
  \label{thm.main} Generalized matrix functions are CSS over the cone of positive semidefinite matrices. In particular, the determinant and permanent are CSS.
\end{thm}

We slightly overload the notation and extract a special case for later use. For any $\bm{A}=[A_{i,j}]_{i,j=1}^m\in \mathbb{M}_m(\mathbb{M}_n)$, define $\mdet_m(\bm{A}) :=[\det A_{i,j}]_{i,j=1}^m$.

\begin{cor} \label{cor1}
  Let  $\bm{A}, \bm{B}\in \mathbb{M}_m(\mathbb{M}_n)$ be  positive semidefinite. Then
  \begin{equation}\label{C=0}
    \mdet_m(\bm{A}+\bm{B}) \ge \mdet_m(\bm{A})+\mdet_m(\bm{B}).
  \end{equation}
  In particular,
  \begin{align*}
    \det\big(\mdet_m(\bm{A}+\bm{B})\big)
    \ge
    \det\big(\mdet_m(\bm{A})\big) + \det\big(\mdet_m(\bm{B})\big).
  \end{align*}
\end{cor}

The proof of  Theorem \ref{thm.main} is given in Section \ref{proof-main result}. In Section \ref{thompson-proof}, we apply Corollary \ref{cor1} to obtain a new proof of a determinantal inequality due to Thompson (1961).

\section{Auxiliary results and proof of Theorem \ref{thm.main}}\label{proof-main result}

We start by recalling standard notation from multilinear algebra \cite{Mar73, Mer97}. Let $\mathcal{V}$ be an $n$-dimensional Hilbert space, and let $\chi$ be a character of degree $1$ on a subgroup $G$ of $S_m$ the symmetric group on $m$ elements. The \emph{symmetrizer} induced by $\chi$ on the tensor product space $\otimes^m \mathcal{V}$ is defined by its action
\begin{equation}
  \label{eq:2}
  S(v_1\otimes \cdots \otimes v_m) := \frac{1}{|G|}\sum_{\sigma \in G} \chi(\sigma) v_{\sigma^{-1}(1)}\otimes \cdots \otimes v_{\sigma^{-1}(m)}.
\end{equation}
Elements of the form~\eqref{eq:2} span a vector space that is denoted as
\begin{equation}
  \label{eq:3}
  \mathcal{V}_\chi^m(G) := S(\otimes^m \mathcal{V}) \subset \otimes^m \mathcal{V}.
\end{equation}
This vector space is the space of the symmetry class of tensors associated with $G$ and $\chi$. It can be verified that $\mathcal{V}_\chi^m(G)$ is an invariant subspace of $\otimes^m \mathcal{V}$. The elements of $\mathcal{V}_\chi^m(G)$ are denoted by the following ``star-product'':
\begin{equation}
  \label{eq:4}
  v_1\star \cdots \star v_m := S(v_1\otimes \cdots \otimes v_m).
\end{equation}
For any linear operator $T$ on $\mathcal{V}$ there is a unique \emph{induced operator}
$K(T) : \mathcal{V}_\chi^m(G) \to \mathcal{V}_\chi^m(G)$ which satisfies (see also \cite{CKLi} for related material):
\begin{equation}
  \label{eq:5}
  K(T)(v_1 \star \cdots \star v_m) =  Tv_1\star \cdots \star Tv_m.
\end{equation}
This operation is usually written as $K(T)v^{\star} = Tv^{\star}$, where $v^{\star} \equiv v_1 \star \cdots \star v_m$.

From an orthonomal basis for $\mathcal{V}$ we can induce an orthonomal basis for $\mathcal{V}_\chi^m(G)$, which will allow us to write down a matrix representation of the operator $K(T)$. To define such a matrix we need some more notation from \cite{Mar73}.

Let $\Gamma_{m,n}$ denote the totality of sequences $\alpha = (\alpha_1,\ldots,\alpha_m)$ such that $1 \le \alpha_i \le n$ for $1\le i \le m$. Thus, $|\Gamma_{m,n}| = n^m$. Two sequences $\alpha$ and $\beta$ in $\Gamma_{m,n}$ are said to be G-equivalent, denoted $\alpha \sim_G \beta$, if there exists a permutation $\sigma \in G$ such that $\alpha = (\beta_{\sigma(1)},\ldots,\beta_{\sigma(n)})$. This equivalence partitions $\Gamma_{m,n}$ into equivalence classes; let $\Delta$ be a system of distinct representatives for these equivalence classes; we order sequences in $\Delta$ using lexicographic order.

For all $\alpha \in \Gamma_{m,n}$ the set of all permutations $\sigma \in G$ for which $\alpha\sigma = \alpha$ is called the \emph{stabilizer} of $\alpha$ and is denoted by $G_\alpha$. Clearly, it is a subgroup of $G$; we denote its order by $\nu(\alpha)$. We define the set $\bar{\Delta} \subset \Delta$ consisting of those $\alpha \in \Delta$ for which $G_{\alpha} \subset \text{ker} \chi$. Since $\chi$ was assumed to be a character of degree 1, $\text{ker} \chi$ is the set of permutations $\sigma$ for which $\chi(\sigma)=1$. Thus, $\alpha \in \bar{\Delta}$ if and only if $\chi(\sigma)=1$ for all $\sigma \in G_{\alpha}$. Therefore,
\begin{equation}
  \label{eq:6}
  \sum_{\sigma \in G_{\alpha}} =
  \begin{cases}
    \nu(\alpha), & \text{if}\ \ \alpha \in \bar{\Delta},\\
    0, & \text{if}\ \ \alpha \not\in \bar{\Delta}.
  \end{cases}
\end{equation}

Now suppose $B = \{e_1,\ldots,e_n\}$ is an orthonomal basis for $\mathcal{V}$. Then,
\begin{equation*}
  B^\star := \{e_{\alpha_1} \star \cdots \star e_{\alpha_m} \mid \alpha \in \bar{\Delta}\},
\end{equation*}
is an orthogonal basis for $\mathcal{V}_\chi^m(G)$, which can be normalized to obtain an orthonomal basis---see e.g., \cite[Theorem 3.2]{Mar73}, which proves that
\begin{equation*}
  \bar{B^\star} = \{ (\sqrt{|G|/\nu(\alpha)})(e_{\alpha_1} \star \cdots \star e_{\alpha_m} \mid \alpha \in \bar{\Delta}),
\end{equation*}
is an orthonormal basis for $\mathcal{V}_\chi^m(G)$ with respect to the induced inner product on $\otimes^m \mathcal{V}$. Moreover, $\dim \mathcal{V}_\chi^m(G) = |\bar{\Delta}|$.

Let $T \in \mathcal{L}(\mathcal{V},\mathcal{V})$. From \cite[Theorem~4.1]{Mar73} we know that $K(T) = \otimes^m T \mid \mathcal{V}_\chi^m(G)$, the \emph{restriction} of the tensor space $\otimes^m T$ to the symmetry class $\mathcal{V}_\chi^m(G)$. Thus, $K(T)v^{\star} = (\otimes^m T)v^{\star}$. Finally, it can be shown that \cite[p. 126]{Mar73} that for multi-indices $\alpha, \beta \in \bar{\Delta}$, the $(\alpha,\beta)$ entry of $K(A)$ is given by
\begin{equation}
  \label{eq:10}
  [K(A)]_{\alpha,\beta} = \frac{1}{\sqrt{\nu(\alpha)\nu(\beta)}}d_{\chi}^G(A^*[\beta|\alpha]),
\end{equation}
where $A^*[\beta|\alpha]$ is the $(\beta,\alpha)$ submatrix of $A^*$. For self-adjoint $A$, we see that we can recover $\dxg(A)$ picking out a diagonal entry of $K(A)$ corresponding to $\beta=\alpha=(1,\ldots,m)$.

With this notation in hand we can state the following easy but key lemma.
\begin{lem}
  \label{lem.repr}
  Let $T \in \mathcal{L}(\mathcal{V},\mathcal{V})$ be a self-adjoint operator with $A$ as its matrix representation. Let $K(T)$ be the induced operator corresponding to the symmetry class described by $\chi$ and subgroup $G \subset S_m$, and let $K(A)$ be the matrix representation of $K(T)$. Then, there exists a matrix $Z$ (of suitable size) such that
  \begin{equation*}
    K(A) = Z^*(\otimes^m A)Z.
  \end{equation*}
\end{lem}
\begin{proof}
  From the discussion above it follows that $[K(A)]_{\alpha,\beta}=\langle K(T)e_{\alpha}^{\star}, e_{\beta}^{\star}\rangle$. Since $K(T)v^{\star}=(\otimes^m T)v^{\star}$, we obtain $[K(A)]_{\alpha,\beta}=\langle (\otimes^m A)e_{\alpha}^{\star},e_{\beta}^{\star}\rangle$. Collecting the vectors $e_{\alpha}^{\star}$ into a suitable matrix $Z$ (note $ZZ^*=I$), we therefore immediately obtain
  \begin{equation*}
    K(A) = Z^*(\otimes^m A)Z.\qedhere
  \end{equation*}
\end{proof}
\noindent Observe that Lemma~\ref{lem.repr} easily yields the well-known multiplicativity of $K$, i.e.,
\begin{equation}
  \label{eq:8}
  K(AB) = K(A)K(B),
\end{equation}
since $\otimes^k(AB)=(\otimes^k A)(\otimes^kB)$ and $ZZ^*=I$.

Next, we refer to the following result from \cite[Lemma 2.2]{TCL11}.
\begin{lem} \label{lem.superadd-tensor} Let $A, B, C\in \mathbb{M}_\ell$ be positive semidefinite.  Then
$$\otimes^k(A + B + C)+\otimes^kC\ge\otimes^k(A+C)+\otimes^k(B+C)$$  for any positive integer $k$.\end{lem}

An immediate corollary of Lemmas~\ref{lem.repr} and \ref{lem.superadd-tensor} is the following.

\begin{cor}
  \label{cor.k}
  Let $A, B, C\in \mathbb{M}_\ell$ be positive semidefinite. Then
  \begin{equation}
    \label{eq:7}
    K(A+B+C) + K(C) \ge K(A+C) + K(B+C).
  \end{equation}
\end{cor}

\begin{lem}\label{compression}
  Let $\bm{A}=[A_{i,j}]_{i,j=1}^m\in\mathbb{M}_m(\mathbb{M}_n)$  be positive semidefinite. Then  the matrix $[K(A_{i,j})]_{i,j=1}^m$ is a compression of the  matrix $K(\bm{A})$.
\end{lem}
\begin{proof}
 We  follow an approach similar to \cite{Zhang12}.   Since $\bm{A}\ge 0$, we can write it as $\bm{A}=R^*R$. Now partition $R=[R_1,\ldots,R_m]$ where each $R_i$, $1 \le i \le m$, is an $mn\times n$ complex matrix. With this partitioning we see that $A_{i,j}=R_i^*R_j$. Also, with this notation, we have $R_i = RE_i$, where $E_i$ is a suitable $mn \times n$ matrix that extracts the $i$th block from $R$.

The crucial property to exploit is the multiplicativity of $K$ and that $K(A^*)=K(A)^*$~\cite[Theorem~4.2]{Mar73}. Consider, thus the block matrix $[K(A_{i,j})]_{i,j=1}^m$. We have
\begin{align*}
  K(A_{i,j}) &= K(R_i^*R_j) = K(E_i^*R^*RE_j)\\
  &= K(E_i)^*K(R^*R)K(E_j) = P_i^*K(\bm{A})P_j.
\end{align*}
In other words,
\begin{equation*}
  [K(A_{i,j})] = \bm{P}^*K(\bm{A})\bm{P},\qquad
  \text{where}\quad
  \bm{P}={\scriptsize\begin{bmatrix} P_1 && \\ &\ddots& \\&&P_m\end{bmatrix}.}\qedhere
\end{equation*}
\end{proof}

 \vspace{0.1in}

\noindent We are now in a position to present a proof of Theorem~\ref{thm.main}.

 \vspace{0.1in}

\noindent\emph{Proof of Theorem \ref{thm.main}.} Let  $\bm{A}=[A_{i,j}]_{i,j=1}^m, \bm{B}=[B_{i,j}]_{i,j=1}^m, \bm{C}=[C_{i,j}]_{i,j=1}^m\in \mathbb{M}_m(\mathbb{M}_n)$ be  positive semidefinite. By Corollary \ref{cor.k},
\begin{equation}
\label{eq:9}
K(\bm{A}+\bm{B}+\bm{C})+K(\bm{C})\ge K(\bm{A}+\bm{C})+K(\bm{B}+\bm{C}).
\end{equation}
By Lemma \ref{compression}, $[K(A_{i,j})]_{i,j=1}^m$ is a compression of  $K(\bm{A})$, which, combined with~\eqref{eq:9} yields the inequality $$[K(A_{i,j}+B_{i,j}+C_{i,j})]_{i,j=1}^m +[K(C_{i,j})]_{i,j=1}^m  \ge[K(A_{i,j}+C_{i,j})]_{i,j=1}^m+[K(B_{i,j}+C_{i,j})]_{i,j=1}^m.$$
Taking into account (\ref{eq:10}), it follows that
$$[\dxg(A_{i,j}+B_{i,j}+C_{i,j})]_{i,j=1}^m +[\dxg(C_{i,j})]_{i,j=1}^m  \ge[\dxg(A_{i,j}+C_{i,j})]_{i,j=1}^m+[\dxg(B_{i,j}+C_{i,j})]_{i,j=1}^m.$$
  therewith establishing the theorem.
\qed

\section{A proof of Thompson's result}\label{thompson-proof}
Thompson \cite{Tho61} proved the following elegant determinantal inequality.
\begin{thm}\label{thompson}
Let $\bm{A}\in\mathbb{M}_m(\mathbb{M}_n)$ be positive semidefinite. Then \begin{eqnarray}\label{thompson-det} \det \bm{A}\le \det \bigl({\det}_m (\bm{A})\bigr). \end{eqnarray} \end{thm}

As an application of our result, we present a new proof of Theorem \ref{thompson}.

 \vspace{0.1in}

\noindent{\it Proof of Theorem \ref{thompson}.} As $\bm{A}\ge 0$, we may write $\bm{A}=\mathbf{T}^*\mathbf{T}$ with $\mathbf{T}=[T_{i,j}]_{i,j=1}^m$ being block upper triangular. If $\bm{A}$ is singular, (\ref{thompson-det}) is trivial. So we assume otherwise. We may further assume $T_{i,i}=I_n$, the $n\times n$ identity matrix, by pre- and post-multiplying both sides of  (\ref{thompson-det}) with $\prod_{i=1}^m\det T_{i,i}^{-*}$ and $\prod_{i=1}^m\det T_{i,i}^{-1}$, respectively. Thus, it suffices to show  \begin{eqnarray}\label{reformulation} \det\bigl({\det}_m(\mathbf{T}^*\mathbf{T})\bigr)\ge 1.  \end{eqnarray}
This reformulation is exactly what Thompson did in \cite{Tho61}.

We prove (\ref{reformulation}) by induction. When $m=2$,
\begin{align*}
  \det\big({\det}_2(\mathbf{T}^*\mathbf{T})\big) &=\det \begin{bmatrix} 1 &  \det T_{1,2}  \\ \det T_{1,2}^*  &   \det(I_n+T_{1,2}^*T_{1,2})\end{bmatrix}\\
  &=\det(I_n+T_{1,2}^*T_{1,2})-\det (T_{1,2}^*T_{1,2})\ge 1.
\end{align*}
Suppose (\ref{reformulation}) is true for $m=k$, and then the case $m=k+1$. For notational convenience, we denote $\mathbf{T}= \begin{bmatrix} I_n &  V\\ 0&\widehat{T}\end{bmatrix}$, where $V= \begin{bmatrix}T_{1,2}&\cdots & T_{1,m}\end{bmatrix}$ and $\widehat{T}=[T_{i+1,j+1}]_{i,j=1}^k$. Let $D=\begin{bmatrix}\det T_{1,2}&\cdots & \det T_{1,m}\end{bmatrix}$. Clearly, $D^*D={\det}_k(V^*V)$.

Now compute
\begin{eqnarray*}
  \mathbf{T}^*\mathbf{T}=
  \begin{bmatrix} I_n &  V\\ 0&\widehat{T}\end{bmatrix}^*
  \begin{bmatrix} I_n &   V\\ 0&\widehat{T}\end{bmatrix}
  = \begin{bmatrix} I_n &  V\\ V^*&\widehat{T}^*\widehat{T}+V^*V\end{bmatrix}.
\end{eqnarray*}
Then \begin{eqnarray*}
  \det\big({\det}_m(\mathbf{T}^*\mathbf{T})\big)&=&
  \det\begin{bmatrix}
    1 &  D\\ D^*& {\det}_k(\widehat{T}^*\widehat{T}+V^*V)
  \end{bmatrix}\\
  &=&\det\big({\det}_k(\widehat{T}^*\widehat{T}+V^*V)-D^*D\big)\\
  &\ge& \det\big({\det}_k (\widehat{T}^*\widehat{T}) +{\det}_k(V^*V)-D^*D\big)\\
  &=&\det\big({\det}_k(\widehat{T}^*\widehat{T})\big)\ge 1, %
\end{eqnarray*}
in which the first inequality is by (\ref{C=0}), while the second one is by the induction hypothesis. This completes the proof. \qed

\subsection*{Acknowledgements}
MathOverflow brought the authors to work on this topic together (see \textsf{\small http://mathoverflow.net/q/173088/}).

\end{document}